\documentclass[12pt]{amsart}

\usepackage{amsmath,amsfonts,amssymb,amsthm,mathabx,shuffle,latexsym,mathtools,mathrsfs}
\usepackage{enumitem}
\usepackage[usenames, dvipsnames, svgnames]{xcolor}
\usepackage{ytableau}
\usepackage{tikz-cd}
\usepackage{pbox}
\usepackage{soul}

\makeatletter
\@namedef{subjclassname@2020}{\textup{2020} Mathematics Subject Classification}
\makeatother

\newtheorem{theorem}{Theorem}[section]
\newtheorem{lemma}[theorem]{Lemma}
\newtheorem{proposition}[theorem]{Proposition}

\newtheorem{procedure}[theorem]{Procedure}

\theoremstyle{definition}
\newtheorem{definition}[theorem]{Definition}
\newtheorem{example}[theorem]{Example}

\theoremstyle{remark}

\numberwithin{equation}{section}

\newcommand{\wt}{\ensuremath{\mathrm{wt}}}

\newcommand{\rev}{\ensuremath{\mathrm{rev}}}

\newcommand{\Des}{\ensuremath{\mathrm{Des}}}

\newcommand{\SPCT}{\ensuremath{\mathsf{SPCT}}}

\newcommand{\MPCT}{\ensuremath{\mathsf{MPCT}}}
\newcommand{\SMPCT}{\ensuremath{\mathsf{SMPCT}}}
\newcommand{\SYCT}{\ensuremath{\mathsf{SYCT}}}
\newcommand{\SPYCT}{\ensuremath{\mathsf{SPYCT}}}
\newcommand{\DIRT}{\ensuremath{\mathsf{DIRT}}}

\newcommand{\SIT}{\ensuremath{\mathsf{SIT}}}

\newcommand{\Std}{\ensuremath{\mathrm{Std}}}
\newcommand{\unmark}{\ensuremath{\mathtt{unmark}}}

\newcommand{\QSym}{\ensuremath{\mathsf{QSym}}}

\newcommand{\rw}{\ensuremath{\mathrm{rw}}}
\newcommand{\Peak}{\ensuremath{\mathrm{Peak}}}

\newcommand{\comp}{\ensuremath{\mathrm{comp}}}

\newcommand{\C}{\ensuremath{\mathbb{C}}}

\newcommand{\set}{\ensuremath{\mathrm{set}}}

\newcommand{\excise}[1]{}

\newlength\cellsize 

\savebox2{%
\begin{picture}(15,15)
\put(0,0){\line(1,0){15}}
\put(0,0){\line(0,1){15}}
\put(15,0){\line(0,1){15}}
\put(0,15){\line(1,0){15}}
\end{picture}}

\newcommand\cellify[1]{\def\thearg{#1}\def\nothing{}%
\ifx\thearg\nothing\vrule width0pt height\cellsize depth0pt%
  \else\hbox to 0pt{\usebox2\hss}\fi%
  \vbox to 15\unitlength{\vss\hbox to 15\unitlength{\hss$#1$\hss}\vss}}

\newcommand\tableau[1]{\vtop{\let\\=\cr
\setlength\baselineskip{-15000pt}
\setlength\lineskiplimit{15000pt}
\setlength\lineskip{0pt}
\halign{&\cellify{##}\cr#1\crcr}}}

\begin{document}


\title[Expanding quasisymmetric Schur $Q$-functions]{Expanding quasisymmetric Schur $Q$-functions into peak Young quasisymmetric Schur functions}

\author[D. Searles]{Dominic Searles}
\address{Department of Mathematics and Statistics, University of Otago, 730 Cumberland St., Dunedin 9016, New Zealand}
\email{dominic.searles@otago.ac.nz}
\author[M. Slattery-Holmes]{Matthew Slattery-Holmes}
\address{Department of Mathematics and Statistics, University of Otago, 730 Cumberland St., Dunedin 9016, New Zealand}
\email{slama077@student.otago.ac.nz}



\date{June 18, 2024}



\begin{abstract}
The dual immaculate and Young quasisymmetric Schur bases of quasisymmetric functions possess analogues in the peak algebra: respectively, the quasisymmetric Schur $Q$-functions and the peak Young quasisymmetric Schur functions. We show elements of the former basis expand into the latter basis with nonnegative coefficients.
\end{abstract}

\maketitle

\section{Introduction}
Several bases of the Hopf algebra of quasisymmetric functions serve as analogues of the famous Schur basis of the Hopf algebra of symmetric functions. These include the widely-studied \emph{fundamental quasisymmetric functions} \cite{Gessel}, \emph{dual immaculate functions} \cite{BBSSZ:dualimmaculate} and \emph{(Young) quasisymmetric Schur functions} \cite{HLMvW11:QS, LMvWbook}.

Both the dual immaculate functions and Young quasisymmetric Schur functions expand in the basis of fundamental quasisymmetric functions with nonnegative coefficients. Theses expansions are indexed by descent sets associated to, respectively, \emph{standard immaculate tableaux} ($\SIT$s) and \emph{standard Young composition tableaux} ($\SYCT$s). 
In \cite{AHM}, Allen, Hallam and Mason gave a formula for the expansion of dual immaculate functions in Young quasisymmetric Schur functions, proving that this expansion also has nonnegative coefficients. To prove their formula they gave a descent-preserving map, defined in terms of an insertion algorithm, that takes an $\SIT$ to an $\SYCT$.

Certain bases of symmetric or quasisymmetric functions have important analogues in the \emph{peak algebra}, a Hopf subalgebra of quasisymmetric functions. 
The \emph{Schur $Q$-functions}, important in the projective representation theory of symmetric groups, form a peak analogue of the Schur functions, and the \emph{peak functions}, introduced in \cite{Stembridge:enriched} in the context of enriched $P$-partitions, form a peak analogue of the fundamental quasisymmetric functions. 

More recently, Jing and Li \cite{Jing.Li} introduced the \emph{quasisymmetric Schur $Q$-functions}, a peak analogue of the dual immaculate functions, and Kantarc{\i} O{\u g}uz \cite{Oguz} proved these expand nonnegatively in peak functions, with the expansion indexed by \emph{standard peak composition tableaux} ($\SPCT$s). Subsequently, Searles \cite{Searles:0HC} introduced the \emph{peak Young quasisymmetric Schur functions}, a peak analogue of the Young quasisymmetric Schur functions, which also expand nonnegatively in peak functions, with the expansion indexed by \emph{standard peak Young composition tableaux} ($\SPYCT$s). 

The question as to whether quasisymmetric Schur $Q$-functions expand into peak Young quasisymmetric Schur functions with nonnegative coefficients was raised in \cite{Searles:0HC}. 
In this paper, we provide a formula for this expansion, giving an affirmative answer to this question. The $\SPCT$s are a subset of the $\SIT$s, and the $\SPYCT$s are a subset of the $\SYCT$s, which suggests the strategy of applying the insertion algorithm of \cite{AHM} to the $\SPCT$s. However, there is a complication in that the notion of descent set of an $\SPCT$ used by Kantarc{\i} O{\u g}uz in \cite{Oguz} differs from that used for an $\SIT$ by Allen, Hallam and Mason in \cite{AHM}. Therefore, in order to apply the techniques of \cite{AHM}, we first obtain a formula for the expansion of a quasisymmetric Schur $Q$-function in peak functions that uses the same notion of descent set as in \cite{AHM}. We also give a procedure that allows the coefficients of the expansion of a quasisymmetric Schur $Q$-function in peak Young quasisymmetric Schur functions to be computed without recourse to insertion, adapting an analogous procedure in \cite{AHM}, and use this to characterise when a quasisymmetric Schur $Q$-function is equal to a peak Young quasisymmetric Schur function.

We recently became aware of work of Choi, Nam and Oh \cite{CNO}, in which our main results (and also the characterisation of when a quasisymmetric Schur $Q$-function is equal to a peak Young quasisymmetric Schur function) are obtained independently. Our approach differs in the method used to obtain the key intermediate result: the formula for quasiymmetric Schur $Q$-functions in terms of peak functions using the same notion of descent as in \cite{AHM}. We obtain this result using a formula in \cite{Oguz} for quasisymmetric Schur $Q$-functions in terms of monomial quasisymmetric functions, whereas in \cite{CNO} this is obtained by constructing a bijection from the set of all $\SPCT$s (of a given shape) to itself that transports one notion of descent set to the other.

\section{Background}
Let $n$ be a positive integer. We denote the set $\{1,\ldots , n\}$ by $[n]$, and for a positive integer $m$ with $m<n$, we denote the set $\{m, m+1, \ldots , n\}$ by $[m,n]$. A \emph{composition} $\alpha$ of $n$, denoted $\alpha \vDash n$, is a finite sequence of positive integers that sum to $n$. If $\alpha = (\alpha_1,\ldots ,\alpha_k)$, we say $\alpha_i$ is the $i$th \emph{part} of $\alpha$, and that the \emph{length} $\ell(\alpha)$ of $\alpha$ is $k$. The \emph{reverse} of a composition $\alpha$, denoted $\alpha^{\rev}$, is the composition obtained by listing the parts of $\alpha$ from right to left.  
A composition $\beta$ \emph{refines} $\alpha$ if $\alpha$ can be obtained by summing consecutive entries of $\beta$. 
If $\alpha$ is a composition of $n$ then $\set(\alpha)$ is the set \[\set(\alpha) = \{\alpha_1,\alpha_1+\alpha_2,\ldots ,\alpha_1+\alpha_2+\cdots +\alpha_{k-1}\}\subseteq [n-1].\] Given a set $X = \{x_1, \ldots , x_r\} \subseteq [n-1]$, where $x_1 < \cdots < x_r$, $\comp_n(X)$ is the composition defined by \[\comp_n(X) = (x_1,x_2-x_1,x_3-x_2,\ldots,x_r-x_{r-1},n-x_r)\vDash n.\]

\begin{example}
$\alpha = (3,2,5,1)$ is a composition of 11 and $l(\alpha)=4$. Here $\alpha^{\rev} = (1,5,2,3)$, the composition $(2,1,2,4,1,1)$ refines $\alpha$, and $\set(\alpha)=\{3,5,10\}$. 
\end{example}

\subsection{Quasisymmetric functions}
Denote by $\C[[x_1,x_2,x_3,\ldots]]$ the Hopf algebra of formal power series of bounded degree in countably many variables. An element $f$ of $\C[[x_1,x_2,x_3,\ldots]]$ is a \emph{quasisymmetric function} if, for any composition $(\alpha_1,\alpha_2,\ldots,\alpha_k)$, the coefficient of $x_1^{\alpha_1}\cdots x_k^{\alpha_k}$ in $f$ is equal to the coefficient of $x_{i_1}^{\alpha_1}\cdots x_{i_k}^{\alpha_k}$ in $f$, for any increasing sequence $i_1<i_2<\cdots<i_k$ of positive integers. Denote the Hopf algebra of quasisymmetric functions by $\QSym$.  

The \emph{monomial quasisymmetric functions} form a basis for $\QSym$. Given any $\alpha = (\alpha_1, \ldots , \alpha_k)\vDash n$, the monomial quasisymmetric function $M_\alpha$ is 
\[M_\alpha = \sum_{1\le i_1<i_2<\cdots <i_k}x_{i_1}^{\alpha_1}x_{i_2}^{\alpha_2}\cdots x_{i_k}^{\alpha_k}.\]

Another basis for $\QSym$ is the \emph{fundamental quasisymmetric functions} \cite{Gessel}. They expand in the monomial quasisymmetric functions with nonnegative coefficients, by the formula
\[F_\alpha = \sum_{\beta \,\, {\rm refines} \,\, \alpha} M_\beta.\] 
Often we will find it useful to index fundamental quasisymmetric functions by sets instead of compositions. For fixed $n$, if $X\subseteq [n-1]$, then we denote $F_{\comp_n(X)}$ by $F_X$.

\begin{example}
The monomial quasisymmetric function corresponding to the composition $(3,1)$ is 
\[M_{(3,1)} = \sum_{1\le i<j} x_i^3x_j.\]
The compositions which refine $(3,1)$ are $(3,1),(2,1,1),(1,2,1),$ and $(1,1,1,1)$, so we have 
\[F_{(3,1)} = M_{(3,1)}+M_{(2,1,1)}+M_{(1,2,1)}+M_{(1,1,1,1)}.\]
\end{example}

Given a subset $X$ of $[n]$, the \emph{peak set of $X$}, denoted $\Peak(X)$, is defined as \[\Peak(X)=\{i\in X:i-1\notin X \text{ and }i\neq 1\}.\] 
For a composition $\alpha$, let $\Peak(\alpha)$ denote $\Peak(\set(\alpha))$. 
A composition $\alpha$ is called a \emph{peak composition} if all parts of $\alpha$ except possibly the last part are greater than $1$. 

The \emph{peak functions} \cite{Stembridge:enriched} form a basis for a Hopf subalgebra of $\QSym$ called the \emph{peak algebra}.

\begin{definition}
Let $\alpha$ be a peak composition of $n$. The peak function $K_\alpha$ is defined by the formula 
    \[K_\alpha = 2^{|\Peak(\alpha)|+1}\sum_{\beta \,\, : \,\,  \Peak(\alpha)\subseteq \set(\beta)\triangle (\set(\beta)+1)}F_\beta,\]
where $\beta\vDash n$, $\set(\beta)+1 = \{i+1:i\in \set(\beta)\}$, and $\triangle$ denotes the symmetric difference of sets. 
\end{definition}

\begin{example}
To expand the peak function $K_{(3,1)}$ into fundamental quasisymmetric functions, note that $\Peak(3,1)=\Peak(\set(3,1)) = \Peak(\{3\}) = \{3\}$, and the compositions $\beta \vDash 4$ such that $\{3\}\subseteq \set(\beta)\triangle (\set(\beta)+1)$ are $(3,1),(2,2),(1,2,1)$, and $(1,1,2)$.  
Therefore, 
\[K_{(3,1)}=4\left(F_{(3,1)}+F_{(2,2)}+F_{(1,2,1)}+F_{(1,1,2)}\right).\]
\end{example}

\section{Quasisymmetric Schur $Q$-functions}\label{sec:peakexpansion} 
In this section we introduce the quasisymmetric Schur $Q$-functions of Jing and Li \cite{Jing.Li}. We then obtain a formula for the expansion of quasisymmetric Schur $Q$-functions into peak functions, in terms of a family of tableaux whose descent sets are the \emph{immaculate descent sets} used in \cite{AHM}. We begin by describing a formula \cite{Oguz} for their expansion in monomial quasisymmetric functions.

The \emph{diagram} of a composition $\alpha$, denoted $D_\alpha$, is the collection of left-justified boxes with $\ell(\alpha)$ rows such that the $i$th row from the bottom contains $\alpha_i$ boxes. We denote the box of $D_\alpha$ in row $r$ and column $c$ by $(c,r)$, where the bottom row is row 1, and the leftmost column is column 1. 

Let $X$ be a set. A \emph{filling} of $D_\alpha$  with entries from $X$ is a function $T:D_\alpha \rightarrow X$. We say $T$ has \emph{shape} $\alpha$, and if $\alpha$ is a peak composition we say $T$ has \emph{peak shape}. We refer to $T(c,r)$ as the \emph{entry} of $T$ in box $(c,r)$, and depict $T$ by writing the entries inside the boxes.  
If $S$ and $T$ are fillings of $D(\alpha)$, then for an entry $T(c,r)$ in $T$, we say the entry $S(c,r)$ is the entry in $S$ \emph{corresponding} to $T(c,r)$.

\begin{definition}{\cite[Definition 3.1]{Oguz}}
Let $\alpha$ be a peak composition of $n$. A filling $T$ of $D_\alpha$ with entries from $1'<1<2'<2<\cdots$ is called a \emph{marked peak composition tableau} if it satisfies the following conditions:
\begin{enumerate}[leftmargin=2cm]
    \item[$\MPCT1$:] Entries in each row weakly increase from left to right, with no repeated marked entries.
    \item[$\MPCT2$:] Entries in the first column strictly increase from bottom to top. 
    \item[$\MPCT3$:] For any $k\le n$, the subdiagram of $D_\alpha$ consisting of all boxes whose entries in $T$ are in $\{1',1,2',2,\ldots ,k', k\}$, is the diagram of a peak composition.
    \item[$\MPCT4$:] If there is an (unmarked) entry $i$ in position $(2,r)$ for some row $r$, then we may not have $i$ or $i'$ in position $(1,r+1)$.  
\end{enumerate}
\end{definition}
Note that in fact the first column of an $\MPCT$ uses only distinct integers, since an incidence of an $i'$ immediately below an $i$ in the first column would violate $\MPCT4$ if the entry immediately right of the $i'$ is $i$, and would violate $\MPCT3$ otherwise.

For $T$ an $\MPCT$, the \emph{weight} $\wt(T)$ of $T$ is the sequence $\big(\wt(T)_1,\wt(T)_2,\ldots ,\wt(T)_n\big)$ where $\wt(T)_i$ denotes the number of occurrences of $i$ or $i'$ in $T$.

For $\alpha$ a peak composition, let $\MPCT(\alpha)$ denote the set of all $\MPCT$ of shape $\alpha$ whose weight has no zero entry to the left of a nonzero entry. In particular, if $T\in \MPCT(\alpha)$ then $\wt(T)$ is a composition.    
For $T\in \MPCT(\alpha)$, we will at times make use of the notation $|T(c,r)| = i$ to mean $T(c,r)\in \{i,i'\}$. 

\begin{example} The filling shown below is an element of $\MPCT(3,5,4,3)$, and has weight $(1,3,1,5,3,2)$.
\[\begin{array}{c@{\hskip \cellsize}c@{\hskip \cellsize}c@{\hskip \cellsize}c@{\hskip \cellsize}c@{\hskip \cellsize}c} \tableau{5'&6&6\\4&4&4&5'\\2'&2&4&4&5'\\1'&2'&3}\end{array}\]
\end{example}

We take the formula of \cite{Oguz} below as our definition of quasiysmmetric Schur $Q$-functions.

\begin{definition}\label{def:qsqf}\cite[Proposition 3.2]{Oguz}
The quasisymmetric Schur $Q$-function $\tilde{Q}_\alpha$ is defined by
\[\tilde{Q}_\alpha = \sum_{T\in \MPCT(\alpha)}M_{\wt(T)}.\]
\end{definition}

The quasisymmetric Schur $Q$-functions form a basis for the peak algebra \cite{Jing.Li}. In \cite{Oguz}, a formula is given for the expansion of $\tilde{Q}_\alpha$ in peak functions, in terms of standard peak composition tableaux. However, in this formula the descent sets associated to these tableaux do not agree with those in \cite{AHM}. We require a new formula for the expansion in terms of peak functions, which uses the same descent sets employed in \cite{AHM}, in order to apply their results in Section~\ref{sec:mainexpansion}. 

Following the approach of \cite{Oguz}, we will use Definition~\ref{def:qsqf} to obtain such a formula. 
The first step is to coarsen the expansion in Definition~\ref{def:qsqf} to an expansion in fundamental quasisymmetric functions.

Given a peak composition $\alpha$ of $n$, define the \emph{standard marked peak composition tableaux} of shape $\alpha$, denoted $\SMPCT(\alpha)$, to be the subset of $\MPCT(\alpha)$ that have each number from $1$ to $n$ as an entry, either marked or unmarked. We now define a standardisation map from $\MPCT(\alpha)$ to $\SMPCT(\alpha)$.

\begin{definition}
Let $\alpha$ be a peak composition of $n$. The \emph{standardisation map} $\Std$ maps $T\in \MPCT(\alpha)$ to a filling $\Std(T)$ as follows.

To begin, let $\Std(T)=D_\alpha$. Supposing $T$ contains $j_1$ instances of $1'$ and $k_1$ instances of $1$, first assign entries from $[j_1]$ to the boxes of $\Std(T)$ which correspond to the boxes of $T$ containing $1'$, in order from bottom to top. Next, assign entries from $[j_1+1, j_1+k_1]$ to the boxes of $\Std(T)$ which correspond to the boxes of $T$ containing $1$, in order from left to right along rows, starting at the highest row containing a $1$ and proceeding downwards. Continue this process with the $j_2$ instances of $2'$ and $k_2$ instances of $2$ in $T$, assigning the entries from $[j_1+k_1+1, j_1+k_1+j_2]$ and then from $[j_1+k_1+j_2+1, j_1+k_1+j_2+k_2]$, and so on with the remaining entries of $T$, until all boxes of $\Std(T)$ have been assigned an entry. 

Mark entries in $\Std(T)$ if and only if the corresponding entry in $T$ is marked.
\end{definition}

\begin{definition}
Given a peak composition $\alpha$ of $n$, for $S\in \SMPCT(\alpha)$, we say $i\in [n-1]$ is a \emph{descent} of $S$ if $i$ is unmarked and strictly below $i+1$ in $S$, or $i+1$ is marked and weakly below $i$ in $S$. The \emph{descent set} of $S$, denoted $\Des(S)$, is the set of descents of $S$. 
\end{definition}
    
\begin{example}\label{ex:standardisation}
Below is an element $T$ of $\MPCT(3,5,4,3)$ and its standardisation. 
\[T = \begin{array}{c@{\hskip \cellsize}c@{\hskip \cellsize}c@{\hskip \cellsize}c@{\hskip \cellsize}c@{\hskip \cellsize}c} \tableau{5'&6&6\\4&4&4&5'\\2'&2&4&4&5'\\1'&2'&3}\end{array}~~\mapsto~~\Std(T) =\begin{array}{c@{\hskip \cellsize}c@{\hskip \cellsize}c@{\hskip \cellsize}c@{\hskip \cellsize}c@{\hskip \cellsize}c} \tableau{13'&14&15\\6&7&8&12'\\3'&4&9&10&11'\\1'&2'&5}\end{array}\]
Observe that $\Std(T)\in \SMPCT(3,5,4,3)$. We have $\Des(\Std(T)) = \{1,5,10\}$.
\end{example}

\begin{proposition}
Let $\alpha$ be a peak composition of $n$. For any $T\in \MPCT(\alpha)$, we have $\Std(T)\in \SMPCT(\alpha)$.
\end{proposition}

\begin{proof}
Let $S$ denote $\Std(T)$. By definition $S$ is of shape $\alpha$. By construction, $S$ contains each number from $[n]$ (marked or unmarked) precisely once, meaning it satisfies $\MPCT4$. It remains to show $S$ satisfies $\MPCT1 - \MPCT3$. 

For $\MPCT1$, we must show entries increase along rows of $S$. Since $T$ satisfies $\MPCT1$, for any $c,r$ we have $T(c,r)\leq T(c+1,r)$. If $|T(c,r)| < |T(c+1,r)|$, then we assign an entry to the box $S(c,r)$ before $S(c+1,r)$, and so $S(c,r)<S(c+1,r)$. If $|T(c,r)|=|T(c+1,r)|$, then both these entries are unmarked ($\MPCT1$). If we have multiple instances of $i$ (unmarked) in a single row of $T$ then we assign entries to the corresponding boxes in $S$ in order from left to right, so again $S(c,r)<S(c+1,r)$. Hence $S$ satisfies $\MPCT1$. 

Similarly, we have $|T(1,r)|< |T(1,r+1)|$ for all rows $r$, so the entry at $S(1,r)$ gets assigned before $S(1,r+1)$. Therefore $S$ satisfies $\MPCT2$.

Suppose for a contradiction that $S$ fails $\MPCT3$. 
Then for some row $r$, we have $S(1,r)<S(1,r+1)<S(2,r)$. Therefore we have $T(1,r)<T(1,r+1)\le T(2,r)$ with $T(1,r+1)= T(2,r)$ only if this entry is unmarked, due to the order of standardisation. If $|T(1,r+1)|< |T(2,r)|$ then $T$ violates $\MPCT3$, if not, $T$ violates $\MPCT4$.
\end{proof}

\begin{lemma}\label{lem:inexpansion}
Let $\alpha$ be a peak composition of $n$ and let $S\in \SMPCT(\alpha)$. For any $T\in \MPCT(\alpha)$ such that $\Std(T)=S$, we have that $M_{\wt(T)}$ appears in the expansion of $F_{\comp_n(\Des(S))}$ in monomial quasisymmetric functions.
\end{lemma}

\begin{proof}
This is equivalent to showing that $\wt(T)$ refines $\comp_n(\Des(S))$. In order to show that this is the case, we need to demonstrate that for any descent in $S$, the corresponding entry $i$ or $i'$ in $T$ is the last $i$ or $i'$ to be assigned to an entry in $S$ under standardisation.

We can consider cases here. We will refer to the assignment of an entry to a box in $S$ which corresponds to a box in $T$ containing entry $i$ as \emph{mapping} $i$ into $S$.

\textbf{Case 1:} First, let us consider mapping an entry $i'$ when there is still another $i'$ yet un-mapped. By definition of standardisation, any subsequent $i'$ is strictly above the current $i'$ in $T$. Hence we map the current $i'$ to some $k'$ in $S$ and the next $i'$ to $(k+1)'$ in $S$. Since $(k+1)'$ is strictly above $k'$ in $S$, $k\notin \Des(S)$.

\textbf{Case 2:} Suppose we map $i'$ into $S$ and the next entry to map into $S$ is an $i$. The $i'$ maps to some $k'$ and the $i$ to $k+1$, but if $k$ is marked and $k+1$ unmarked in $S$, then $k\notin \Des(S)$.

\textbf{Case 3:} Suppose we map some $i$ and the next number to map is also $i$. In this situation, by the ordering in which we map these numbers, the second $i$ must be weakly below the first $i$. In this case, the first $i$ maps to some $k$ and the next $i$ to $k+1$; since $k+1$ is weakly below $k$ in $S$, $k\notin \Des(S)$. 
\end{proof} 

\begin{lemma}\label{lem:destandardise}
Let $\alpha$ be a peak composition of $n$, let $S\in \SMPCT(\alpha)$, and let $\beta$ be a composition that refines $\comp_n(\Des(S))$. There is a unique $T\in \MPCT(\alpha)$ with $\wt(T)=\beta$ such that $\Std(T)=S$.
\end{lemma}
\begin{proof} 
Construct a filling $T$ of $D(\alpha)$ by assigning $\beta_1$ $1$'s to the boxes of $T$ that correspond to entries in $S$ less than or equal to $\beta_1$, then $\beta_2$ $2$'s to the boxes of $T$ that correspond to the entries in $S$ from $[\beta_1+1,\beta_1+\beta_2]$, and so on, marking an entry in $T$ if and only if the corresponding entry in $S$ is marked; see Example~\ref{ex:destandardise}. 
Note that since $\beta$ refines $\comp_n(\Des(S))$, there cannot be any descent of $S$ in the interval of entries in $S$ corresponding to $i$ or $i'$ in $T$, for any $i$. 
We need to show $T\in \MPCT(\alpha)$, $\Std(T)=S$, and that no other element of $\MPCT(\alpha)$ with weight $\beta$ standardises to $S$.

First, we show $T\in \MPCT(\alpha)$. 
Since entries increase along rows of $S$, entries in a row of $T$ are assigned only after all entries to the left in the same row of $T$ have already been assigned. Therefore entries of $T$ weakly increase along rows. We need to verify that $T$ does not have two instances of some marked entry $j'$ in the same row. Suppose for a contradiction that $S$ has marked entries $p'<q'$ in the same row, both corresponding to $j'$ in $T$. We claim there is a descent of $S$ in the interval $[p,q-1]$. If $p=q-1$ then $p\in \Des(S)$. Suppose $p < q-1$. If $q-1\notin \Des(S)$, it must be marked and strictly below $q'$. Then for $q-2$ to not be a descent it must be marked and strictly below $(q-1)'$, and so on. So if there are no descents of $S$ in $[p+1,q-1]$, then $p+1$ is marked in $S$ and is below $p'$, hence $p'\in \Des(S)$. In any case, there is a descent of $S$ in $[p,q-1]$, giving the desired contradiction. 
Therefore $T$ satisfies $\MPCT1$.

Suppose for a contradiction that $T$ fails $\MPCT2$. 
Since entries increase up the first column of $S$, the entries in the first column of $T$ are assigned in order from bottom to top, so $T$ must have at least two (necessarily adjacent) entries from $\{i,i'\}$ in the first column. Choose two such adjacent entries, and let $p,q$ denote the corresponding entries of $S$ (with $p$ below $q$); note $p,q$ could be marked or unmarked. This configuration can occur in $T$ only if $S$ has no descent in $[p,q-1]$. All entries of $S$ in $[p,q-1]$ also correspond to $i$ or $i'$ in $T$, so there can be no descent of $S$ in $[p,q-1]$.
If $p=q-1$, then $S$ violates $\MPCT3$ on the entries at most $p$, so assume $p<q-1$. If $p$ is unmarked in $S$, then to avoid a descent in $[p,q-1]$, $p+1$ must be unmarked and weakly below $p$, then $p+2$ must be unmarked and weakly below $p+1$, and so on, but this implies $q-1$ is unmarked and strictly below $q$, and so $q-1\in \Des(S)$ regardless of whether $q$ is marked or unmarked. If $p$ is marked in $S$, consider the entry $x$ immediately right of $p$ (which exists, by $\MPCT3$). If $x$ is marked, then $q'$ and $x'$ are in the same row and both correspond to $i'$ in $T$, which the argument for $T$ satisfying $\MPCT1$ has already shown is impossible. Hence assume $x$ is unmarked. For $x$ to not be a descent, $x+1$ must be unmarked and weakly below $x$, then $x+2$ must be unmarked and weakly below $x+1$, and so on, and as before this implies $q-1$ is unmarked and strictly below $q$, so $q-1\in \Des(S)$. Hence $S$ has a descent in $[p,q-1]$, a contradiction. Thus $T$ satisfies $\MPCT2$.

Suppose for a contradiction that $T$ fails $\MPCT3$. Then we have $|T(1,r_1)|<|T(1,r_2)|<|T(2,r_1)|$ for some $r_1<r_2$; the first inequality is strict since it occurs in the first column, and the second inequality must be strict in order to violate $\MPCT3$. This implies the same relationship between the corresponding entries in $S$, so $S$ violates $\MPCT3$, a contradiction. 
Hence $T$ satisfies $\MPCT3$.

Finally, suppose for a contradiction that $T$ fails $\MPCT4$. Then for some $r$, $T(2,r)$ is an unmarked $i$ and $T(1,r+1)$ is an $i$ or $i'$. We will show there is a descent in the corresponding interval of entries in $S$. 
Let us denote $|S(1,r+1)| = q$ and $S(2,r) = p$. We know $p<q$ since otherwise $S$ violates $\MPCT3$. Also, it must be the case that $q-1$ lies strictly below $q$ in $S$, since entries of $S$ increase along rows and up the first column. 
If $q-1$ is unmarked then $q-1\in \Des(S)$. This includes the case $p=q-1$. So assume $(q-1)$ is marked, which implies $p<q-1$ since $p$ is unmarked.   
Let $x$ be the smallest marked element in $[p,q-1]$; $x$ exists since $q-1$ is marked, and $x-1$ is in $[p,q-1]$ since $p$ is unmarked. We have $x-1\in \Des(S)$, contradicting our assumption that $\beta$ refines $\comp_n(\Des(S))$. Hence $T$ must satisfy $\MPCT4$.
 
Next, we verify that $\Std(T)=S$. Given $k$ and $k+1$ in $S$, marked or unmarked, which both correspond to entries in $\{j,j'\}$ in $T$, we need to confirm that the $j$ which corresponds to $k$ in $S$ is assigned first. 
We can proceed by cases, depending on if $k$ and $k+1$ are marked or unmarked. 

\textbf{Case 1: $k',k+1\in S$.} Here $k'$ will be assigned before $k+1$, since marked entries standardise before unmarked entries. 

\textbf{Case 2: $k',k+1'\in S$.} If $k+1'$ is weakly below $k'$, then $k\in \Des(S)$, contradicting that $\beta$ refines $\comp_n(\Des(S))$. Hence $k+1'$ is above $k'$ and so $k'$ will be assigned to $S$ before $k+1'$. 

\textbf{Case 3: $k,k+1'\in S$.} Here $k\in \Des(S)$, contradicting that $\beta$ refines $\comp_n(\Des(S))$, so this case cannot occur. 

\textbf{Case 4: $k,k+1\in S$.} If $k+1$ is strictly above $k$ in $S$, then $k\in \Des(S)$. Therefore $k+1$ is weakly below $k$, meaning $k+1$ is either to the right of $k$ in the same row, or strictly below $k$. In either situation, the entry of $T$ in the position corresponding to $k$ in $S$ will standardise before that corresponding to $k+1$.

For uniqueness of $T$, suppose $U\in \MPCT(\alpha)$ satisfies $\wt(U)=\wt(T)$ and $Std(U) = S$. For any $i$, the entries in $S$ which the standardisation procedure assigns corresponding to the entries $i$ in $U$ are strictly smaller than the entries in $S$ which the standardisation procedure assigns corresponding to the entries $i+1$ in $U$. Therefore, the $1$'s must be in the same boxes in both $U$ and $T$, then the $2$'s must be in the same boxes in both $U$ and $T$, and so on, and an entry is marked in $U$ if and only if it is marked in $T$, since the marked entries are determined by $S$. Hence $U=T$.
\end{proof}

\begin{example}\label{ex:destandardise}
Let $S$ be the $\SMPCT$ of shape $(3,5,4,3)$ shown below and let $\beta = (1,3,1,5,3,2)$. Note $\Des(S)=\{1, 5,10\}$. The construction in the first paragraph of the proof of Lemma~\ref{lem:destandardise} yields $T\in \MPCT(3,5,4,3)$ shown below (cf. Example~\ref{ex:standardisation}). 
\[S \, = \, \begin{array}{c@{\hskip \cellsize}c@{\hskip \cellsize}c@{\hskip \cellsize}c@{\hskip \cellsize}c@{\hskip \cellsize}c} \tableau{13'&14&15\\6&7&8&12'\\3'&4&9&10&11'\\1'&2'&5}\end{array} \qquad \qquad \qquad T \, = \, \begin{array}{c@{\hskip \cellsize}c@{\hskip \cellsize}c@{\hskip \cellsize}c@{\hskip \cellsize}c@{\hskip \cellsize}c} \tableau{5'&6&6\\4&4&4&5'\\2'&2&4&4&5'\\1'&2'&3}\end{array}\]
\end{example}

\begin{proposition}\label{prop:Fexpansion}
Let $\alpha$ be a peak composition of $n$. We have 
\[\tilde{Q}_\alpha = \sum_{S\in \SMPCT(\alpha)}F_{\comp_n(\Des(S))}.\]
\end{proposition}
\begin{proof}
Combining Lemmas~\ref{lem:inexpansion} and \ref{lem:destandardise}, we obtain 
\[\sum_{T \, \in \, \MPCT(\alpha) \, : \, \Std(T)=S}M_{\wt(T)}=F_{\comp_n(\Des(S))},\]
from which the statement follows.
\end{proof}

We now introduce the family of tableaux used in \cite{Oguz} for the expansion of quasisymmetric Schur $Q$-functions into peak functions. 

\begin{definition}\cite{Oguz}
Let $\alpha$ be a peak composition of $n$. A filling $T$ of $D_\alpha$ with entries from $[n]$, each used once, is called a \emph{standard peak composition tableau} ($\SPCT$) of shape $\alpha$ if it satisfies the following conditions:
\begin{enumerate}[leftmargin=2cm]
    \item[$\SPCT1$:] Entries in each row increase from left to right.
    \item[$\SPCT2$:] Entries in the first column increase from bottom to top.
    \item[$\SPCT3$:] For all $k\le n$, the subdiagram of $D_\alpha$ consisting of the boxes that contain entries less than or equal to $k$ in $T$, is the diagram of a peak composition. 
\end{enumerate}
\end{definition}

Denote the set of standard peak composition tableaux of shape $\alpha$ by $\SPCT(\alpha)$. Observe that $\SPCT(\alpha)$ is precisely the subset of $\SMPCT(\alpha)$ with no marked entries.

\begin{definition}\label{def:DesSPCT}
For $T\in \SPCT(\alpha)$, the \emph{descent set} $\Des_{\uparrow}(T)$ of $T$ is given by \[\Des_{\uparrow}(T)=\{i\in T:i+1\text{ is strictly above } i\}.\]
\end{definition}
We use the upward arrow to indicate that a descent is an instance of $i+1$ above $i$. This definition of descent set is the same as the \emph{immaculate descent set} used in \cite{AHM}, and is different from the descent set defined for an $\SPCT$ in \cite{Oguz}. 

\begin{example}\label{SPCTex}
    For the peak composition $(3,3)$ we have 
\[\SPCT(3,3) = \left\{\,\, \begin{array}{c@{\hskip \cellsize}c@{\hskip \cellsize}c@{\hskip \cellsize}c@{\hskip \cellsize}c@{\hskip \cellsize}c} \tableau{4&5&6\\1&2&3}~,~ \tableau{3&5&6\\1&2&4}~,~\tableau{3&4&6\\1&2&5}~,~ \tableau{3&4&5\\1&2&6}\end{array}\!\!\!\right\}\] 
    with descent (and peak) sets respectively $\{3\},\{2,4\},\{2,5\},\{2\}$.
\end{example}

For $S\in \SMPCT(\alpha)$, define $\unmark(S)$ to be the $\SPCT$ obtained by removing all marks from entries of $S$. 
Given $Q\in \SPCT(\alpha)$, we denote $\Peak(\Des_{\uparrow}(Q))$ by $\Peak_{\uparrow}(Q)$.

\begin{lemma}\label{lem:unmark}
Let $Q\in \SPCT(\alpha)$. We have 
\[\sum_{\substack{S \, : \, \unmark(S)=Q\\S\in \SMPCT(\alpha)}}F_{\comp_n(\Des(S))} = K_{\comp_n(\Peak_{\uparrow}(Q))}.\]
\end{lemma}
\begin{proof}
We will first show that if $S\in \SMPCT(\alpha)$ and $\unmark(S)=Q$, then 
$\Peak_{\uparrow}(Q)\subseteq \Des(S) \triangle (\Des(S)+1)$. 
Secondly, we will show that for any set $D\subset [n-1]$ such that $\Peak_{\uparrow}(Q)\subseteq D\triangle (D+1)$, there are precisely $2^{|\Peak_{\uparrow}(Q)|+1}$ elements $S$ of $\SMPCT(\alpha)$ such that $\unmark(S)=Q$ and $\Des(S)=D$. 
For the first statement, suppose $i\in \Peak_{\uparrow}(Q)$. Then $i\in \Des_{\uparrow}(Q)$ and $i-1\notin \Des_{\uparrow}(Q)$, so $i$ is strictly below $i+1$ and weakly above $i$. Therefore, if $i$ is marked in $S$ then $i\notin \Des(S)$ and $i-1\in \Des(S)$, whereas if $i$ is unmarked in $S$, then $i\in \Des(S)$ and $i-1\notin \Des(S)$. Either way, $i\in \Des(S) \triangle (\Des(S)+1)$.

For the second statement, suppose $\Peak_{\uparrow}(Q) = \{j_1 < \cdots < j_r\}$. Let $D$ be a subset of $[n-1]$ such that $\Peak_{\uparrow}(Q)\subseteq D\triangle (D+1)$. We will construct each $S\in \SMPCT(\alpha)$ such that $\unmark(S)=Q$ and $\Des(S)=D$. Begin with $Q$. 

Let $1\le k <r$ and consider the entries of $Q$ from $[j_k+1, j_{k+1}]$. By the definition of peak set, there is some $0\le h \le j_{k+1}-2$ such that the first $h$ entries in $[j_k+1, j_{k+1}]$ are descents of $Q$, and the remaining entries are not descents of $Q$. For entries $p$ of $Q$ such that $p\in [j_k+1, h]$, mark $p$ if and only if $p\notin D$. For entries $p$ of $Q$ such that $p\in [h+1, j_{k+1}-1]$, mark $p+1$ if and only if $p\in D$. For $h+1$ itself, we may choose whether to mark it or leave it unmarked. 

We now confirm that for $\Des(S)$ to agree with $D$ on the entries from $[j_k+1, j_{k+1}]$, this choice of marking was forced for all entries except $h+1$, where it was optional. For entries $p$ of $Q$ such that $p\in [j_k+1, h]$, we have $p$ strictly below $p+1$. Therefore, $p\in \Des(S)$ if and only if $p$ is unmarked. For entries $p$ of $Q$ such that $p\in [h+1, j_{k+1}-1]$, we have $p$ weakly above $p+1$, and therefore $p\in \Des(S)$ if and only if $p+1$ is marked. In $Q$, the $h+1$ appears strictly above $h$ and weakly below $h+2$. Therefore, marking or unmarking $h$ does not affect whether $h$ or $h+1$ are descents of $S$, regardless of whether $h$ or $h+2$ are marked. 
We also confirm $j_{k+1}\in \Des(S)$ if and only if $j_{k+1}\in D$. Note that $j_{k+1}$ is in $\Des_{\uparrow}(Q)$. If $j_{k+1}-1\in D$ then we have marked $j_{k+1}$, which implies that $j_{k+1}\notin \Des(S)$. However, since $\Peak_{\uparrow}(Q)\subseteq D\triangle (D+1)$, if $j_{k+1}-1\in D$ then $j_{k+1}\notin D$, as required. If on the other hand $j_{k+1}-1\notin D$ then we have left $j_{k+1}$ unmarked, which implies $j_{k+1}\in \Des(S)$. But since $j_{k+1}-1\notin D$ we have $j_{k+1}\in D$, as required.

Apply the same rule to the entries of $Q$ from $[j_r+1, n]$ and also to those from $[1, j_1-1]$. Note that $1\notin \Des_{\uparrow}(Q)$ by $(\SPCT3)$, and thus none of the entries from $[1, j_1-1]$ are in $\Des_{\uparrow}(Q)$, so we take $h=1$ in this case. Since there are $|\Peak_{\uparrow}(Q)|+1$ entries that we can choose to be marked or unmarked, there are precisely $2^{|\Peak_{\uparrow}(Q)|+1}$ tableaux $S\in \SMPCT(\alpha)$ that satisfy $\Des(S)=D$ and $\unmark(S)=Q$.  
\end{proof}

As an immediate consequence of Proposition~\ref{prop:Fexpansion} and  Lemma~\ref{lem:unmark}, we obtain our desired formula. 

\begin{theorem}\label{qsqfpeaks}
Let $\alpha$ be a peak composition of $n$. Then
\[\tilde{Q}_\alpha = \sum_{T\in \SPCT(\alpha)}K_{\comp_n(\Peak_{\uparrow}(T))}.\]
\end{theorem}

\begin{example}
Building on Example \ref{SPCTex}, we have \[\tilde{Q}_{(3,3)} = K_{(3,3)}+K_{(2,2,2)}+K_{(2,3,1)}+K_{(2,4)}.\]
\end{example}

\section{Expanding into peak Young quasisymmetric Schur functions}\label{sec:mainexpansion}
The main goal of this section is to prove a formula for the expansion of quasisymmetric Schur $Q$-functions into the peak Young quasisymmetric Schur functions defined in \cite{Searles:0HC}. We begin by introducing these functions.

\begin{definition}\cite[Definition 5.6]{Searles:0HC}
Let $\alpha$ be a peak composition of $n$. A \emph{standard peak Young composition tableau} ($\SPYCT$) of shape $\alpha$ is a filling $T$ of $D_\alpha$ with entries from $[n]$, each used once, satisfying
\begin{enumerate}[leftmargin=2.3cm]
    \item[$\SPYCT1$:] Entries in each row increase from left to right.
    \item[$\SPYCT2$:] Entries in the first column increase from bottom to top.
    \item[$\SPYCT3$:] For all $k\le n$, the subdiagram of $D_\alpha$ consisting of the boxes that contain entries less than or equal to $k$ in $T$, is the diagram of a peak composition.
    \item[$\SPYCT4$:] If $T(c,r)<T(c+1,r')$ for some $c$ and some $r'<r$, then box $(c+1,r)$ must be in $T$, and $T(c+1,r)<T(c+1,r')$. 
\end{enumerate}
Let $\SPYCT(\alpha)$ denote the set of all standard peak Young composition tableaux of shape $\alpha$.
\end{definition}

\begin{definition}\label{def:DesSPYCT}
For $T\in \SPYCT(\alpha)$, the descent set $\Des_{\leftarrow}(T)$ of $T$ is the set \[\Des_{\leftarrow}(T) = \{i\in T:i+1\text{ is weakly left of }i\text{ in }T\}.\]
\end{definition}
 Note the difference between this notion of descent set and the descent sets for elements of $\SPCT(\alpha)$ (Definition~\ref{def:DesSPCT}). 

We use the following formula as our definition for peak Young quasisymmetric Schur functions.

\begin{definition}\cite[Theorem 5.9]{Searles:0HC}\label{def:PYQS}
Let $\alpha$ be a peak composition of $n$. The \emph{peak Young quasisymmetric Schur function} $\tilde{S}_\alpha$ is defined by
\[\tilde{S}_\alpha = \sum_{T\in \SPYCT(\alpha)}K_{\comp_n(\Peak_{\leftarrow}(T))},\]
where $\Peak_{\leftarrow}(T)$ denotes $\Peak(\Des_{\leftarrow}(T))$.  
\end{definition}

The peak Young quasisymmetric Schur functions form a basis for the peak algebra \cite[Theorem 5.11]{Searles:0HC}.

\begin{example}
Let $\alpha = (3,3) \vDash 6$. We have
\[\SPYCT(3,3) = \left\{  \begin{array}{c@{\hskip \cellsize}c@{\hskip \cellsize}c@{\hskip \cellsize}c@{\hskip \cellsize}c@{\hskip \cellsize}c} \tableau{4&5&6\\1&2&3}~,~\tableau{3&5&6\\1&2&4}~,~\tableau{3&4&5\\1&2&6}  \end{array}\!\!\!\!\!\right\}.\]
These have descent (and peak) sets $\{3\},\{2,4\}$, and $\{2,5\}$ respectively. Therefore 
\[\tilde{S}_{(3,3)} = K_{(3,3)}+K_{(2,2,2)}+K_{(2,3,1)}.\]
\end{example}

In order to prove a formula for expanding quasisymmetric Schur $Q$-functions into peak Young quasisymmetric Schur functions, we will make use of an insertion algorithm defined by Allen, Hallam and Mason \cite{AHM}.

Let $\alpha$ be a composition of $n$. A \emph{standard immaculate tableau} ($\SIT$) of shape $\alpha$ is a filling of $D_\alpha$ with entries from $[n]$, each used once, that satisfies $\SPCT1$ and $\SPCT2$. A \emph{standard Young composition tableau} ($\SYCT$) of shape $\alpha$ is a filling of $D_\alpha$ with entries from $[n]$, each used once, that satisfies $\SPYCT1$, $\SPYCT2$ and $\SPYCT4$. Let $\SIT(\alpha)$ and $\SYCT(\alpha)$ denote the sets of, respectively, standard immaculate tableaux and standard Young composition tableaux of shape $\alpha$.

In particular, when $\alpha$ is a peak composition, we have $\SPCT(\alpha)\subseteq \SIT(\alpha)$ and $\SPYCT(\alpha)\subseteq \SYCT(\alpha)$. 

The (immaculate) descent set $\Des_\uparrow(T)$ of $T\in \SIT(\alpha)$ is defined identically to the descent set of an element of $\SPCT(\alpha)$ (Definition~\ref{def:DesSPCT}), and the descent set $\Des_\leftarrow(T)$ of $T\in \SYCT(\alpha)$ is defined identically to the descent set of an element of $\SPYCT(\alpha)$ (Definition~\ref{def:DesSPYCT}). The \emph{dual immaculate function} $\mathfrak{S}_\alpha$ \cite{BBSSZ:dualimmaculate} and \emph{Young quasisymmetric Schur function} $\mathcal{S}_\alpha$ \cite{HLMvW11:QS, LMvWbook} are defined by
\[\mathfrak{S}_\alpha = \sum_{T\in \SIT(\alpha)}F_{\comp_n(\Des_{\uparrow}(T))} \qquad \mbox{ and } \qquad \mathcal{S}_\alpha = \sum_{T\in \SYCT(\alpha)}F_{\comp_n(\Des_{\leftarrow}(T))}.\]

A formula for the expansion of dual immaculate functions in Young quasisymmetric Schur functions was obtained in \cite[Theorem 1.1]{AHM} using the insertion algorithm described below; in particular, this insertion can be used to map an $\SIT$ to an $\SYCT$ in a way that preserves the descent set.  
For any filling $T$ of $D(\alpha)$ with entries from $\{1,2, \ldots \}$, the \emph{augmented diagram} $\Bar{T}$ of $T$ is the filling obtained by adjoining an additional box to the rightmost end of each row, and placing $\infty$ in this box. 

\begin{procedure}\cite[Procedure 3.1]{AHM}\label{proc:insert}
Let $T$ be a filling of $D_\alpha$ with entries from $\{1,2, \ldots \}$ in which entries weakly increase along rows, strictly increase up the first column, and satisfy $\SPYCT4$. Let $k$ be a positive integer that does not appear in $T$. Let $(c_1,d_1),(c_2,d_2),\ldots $ be the boxes of $\Bar{T}$ listed in order from the highest box in the rightmost column to the lowest box in the leftmost column, proceeding down each column from right to left. 
Set $k_0\coloneqq k$, and let $i$ be the smallest positive integer such that $\Bar{T}(c_i-1,d_i)\le k_0<\Bar{T}(c_i,d_i)$. If $i$ exists, there are two cases. 

\textbf{Case 1:} If $\Bar{T}(c_i,d_i)=\infty$,  place $k_0$ in box $(c_i,d_i)$ and terminate the procedure.

\textbf{Case 2:} If $\Bar{T}(c_i,d_i) \neq \infty$, set $k:= \Bar{T}(c_i,d_i)$, place $k_0$ in box $(c_i,d_i)$, and repeat the procedure by inserting $k$ into the sequence of boxes $(c_{i+1},d_{i+1}),(c_{i+2},d_{i+2}),\dots$. We say that $\Bar{T}(c_i,d_i)$ is \emph{bumped}. 

If no such $i$ exists, begin a new row containing only $k_0$ in the highest position in the first column such that all other entries in the first column below $k_0$ are smaller than $k_0$, and terminate the procedure. If the row created is not the top row of the diagram, shift all higher rows up by 1.
\end{procedure}

We will only apply Procedure~\ref{proc:insert} to fillings $T$ that satisfy $\SPYCT1$, $\SPYCT2$ and $\SPYCT4$. The filling resulting from such an application of Procedure~\ref{proc:insert} will also satisfy these conditions, as noted in \cite{AHM}.

\begin{example}
    Below we demonstrate inserting the entry $10$ into an augmented diagram using Procedure~\ref{proc:insert}.

    \[10 \longrightarrow \begin{array}{c@{\hskip \cellsize}c@{\hskip \cellsize}c@{\hskip \cellsize}c@{\hskip \cellsize}c@{\hskip \cellsize}c}\tableau{8&9&11&\infty\\5&6&12&\infty\\3&4&\infty}\end{array} = \,\, \begin{array}{c@{\hskip \cellsize}c@{\hskip \cellsize}c@{\hskip \cellsize}c@{\hskip \cellsize}c@{\hskip \cellsize}c}\tableau{8&9&10\\5&6&11\\3&4&12}\end{array} \]
\end{example}

For $T\in \SIT(\alpha)$, the \emph{reading word} $\rw(T)$ of $T$ is the sequence obtained by reading the entries of $T$ from left to right in each row, reading the rows in order from top to bottom. In \cite[Section 3]{AHM}, Allen, Hallam and Mason use their insertion algorithm to define a map that takes (the reading word of) a standard immaculate tableau $T$ to a pair consisting of an $\SYCT$ $P$ and a recording tableau $Q$, which we reproduce here and call \emph{reading insertion}. 

\begin{definition}[\emph{Reading insertion}, \cite{AHM}]
 Begin with $(P,Q)=(\emptyset,\emptyset)$, where $\emptyset$ is the empty filling. Let  $\rw(T)=k_1k_2 \cdots k_n$. Insert $k_1$ into $P$ using the insertion defined in Procedure~\ref{proc:insert}, and let $P_1$ be the resulting filling. Record the location where the new box was created by placing a $1$ in $Q$ in the corresponding position, and let $Q_1$ be the resulting filling. Now suppose that the first $j-1$ letters of $\rw(T)$ have been inserted. Insert $k_j$ into $P_{j-1}$ and let $P_j$ be the resulting filling. Place the letter $j$ into the box of $Q_{j-1}$ that corresponds to the box created in $P_j$ by this insertion, and let $Q_j$ be the resulting filling. Once $P_n$ and $Q_n$ are computed, set $P=P_n$ and $Q=Q_n$. The \emph{reading insertion} of $T$ is the pair $(P,Q)$.
\end{definition}

The recording tableau $Q$ produced by the reading insertion of a standard immaculate tableau is called a dual immaculate recording tableau (\DIRT). In \cite{AHM}, $\DIRT$s are characterised as follows. A \emph{row strip} of length $j$ in $Q$ is a maximal increasing sequence $a_1,a_2,\ldots,a_j$ of $j$ consecutive integers such that for all $1\leq i  < j$, $a_{i+1}$ appears strictly right of $a_i$ in $Q$.

\begin{definition}\cite[Definition 3.9]{AHM}
Let $\alpha$ be a composition of $n$. A filling $Q$ of $D_\alpha$ with entries from $[n]$, each used once, is a \emph{dual immaculate recording tableau} ($\DIRT$) if it satisfies the following properties:
\begin{enumerate}[leftmargin=2cm]
    \item[$\DIRT1$:] Entries in each row increase from left to right.
    \item[$\DIRT2$:] The row strips start in the first column.
    \item[$\DIRT3$:] Entries in the first column increase from top to bottom. 
    \item[$\DIRT4$:] If $Q(c,r)>Q(c,r')$ for some $c$ and some $r'<r$, then $Q$ has an entry in box $(c+1,r')$, and $Q(c,r)>Q(c+1,r')$.
\end{enumerate}
\end{definition}

The \textit{row strip shape} of a $\DIRT$ $Q$ is the composition $(\alpha_1,\ldots ,\alpha_k)$ where $\alpha_i$ is the length of the row strip that starts with $\alpha_1+\alpha_2+\cdots +\alpha_{i-1}+1$.

\begin{example}
The filling $Q$ shown below is a $\DIRT$ with row strips $\{1,2\},\{3,4,5\},\{6,7,8,9\}$, and so has row strip shape $(2,3,4)$.
\[Q=\begin{array}{c@{\hskip \cellsize}c@{\hskip \cellsize}c@{\hskip \cellsize}c@{\hskip \cellsize}c@{\hskip \cellsize}c}\tableau{1&2\\3&4&5&8&9\\6&7}\end{array}\]
\end{example}

\begin{example}\label{expansion} 
We illustrate the reading insertion map by applying it to the elements of $\SPCT(3,2)$. We have
\[\SPCT(3,2) = \left\{ T_1 =\tableau{4&5\\1&2&3}~,~T_2 = \tableau{3&5\\1&2&4}~,~T_3 = \tableau{3&4\\1&2&5} \right\}\]
The reading insertion of these tableaux gives  
    \begin{align*}
    \rw(T_1) = 45123 \mapsto\left(\begin{array}{c@{\hskip \cellsize}c@{\hskip \cellsize}c@{\hskip \cellsize}c@{\hskip \cellsize}c@{\hskip \cellsize}c}\tableau{4&5\\1&2&3} ~,~ \tableau{1&2\\3&4&5} \end{array}\!\!\!\!\!\right)
    \\
    \rw(T_2) = 35124 \mapsto\left(\begin{array}{c@{\hskip \cellsize}c@{\hskip \cellsize}c@{\hskip \cellsize}c@{\hskip \cellsize}c@{\hskip \cellsize}c}\tableau{3&5\\1&2&4} ~,~ \tableau{1&2\\3&4&5} \end{array}\!\!\!\!\!\right)
    \\
    \rw(T_3) = 34125 \mapsto\left(\begin{array}{c@{\hskip \cellsize}c@{\hskip \cellsize}c@{\hskip \cellsize}c@{\hskip \cellsize}c@{\hskip \cellsize}c}\tableau{3&4&5\\1&2} ~,~ \tableau{1&2&5\\3&4} \end{array}\!\!\!\!\!\right)
    \end{align*}
Notice that in each case, the filling $P$ created by insertion of the reading word of an $\SPCT$ is in fact an $\SPYCT$. Shortly, we will show this is true in general.
\end{example}

\begin{theorem}\cite{AHM}\label{thm:insertionofT}
Let $\alpha$ be a composition of $n$. The reading insertion map is a bijection between $\SIT(\alpha)$ and pairs $(P,Q)$ where $P$ is an $\SYCT$, $Q$ is a $\DIRT$ with row strip shape $\alpha^{\rev}$, and $P$ and $Q$ have the same shape. Moreover, if $P$ is the $\SYCT$ associated to $T\in \SIT(\alpha)$, then $\Des_{\uparrow}(T) = \Des_{\leftarrow}(P)$.
\end{theorem}

Now we restrict the domain of this bijection to $\SPCT(\alpha)$. Note that since the bijection preserves descent sets, it also preserves peak sets. We now state the peak analogue of Theorem~\ref{thm:insertionofT}, the proof of which will rely on Propositions \ref{prop:TisSPYCT} and \ref{prop:notspctnotspyct}.

\begin{theorem}\label{bijection}
Let $\alpha$ be a peak composition of $n$. The reading insertion map is a bijection between $\SPCT(\alpha)$ and pairs $(P,Q)$ where $P$ is an $\SPYCT$, $Q$ is a $\DIRT$ with row strip shape $\alpha^{\rev}$, and $P$ and $Q$ have the same shape. Moreover, if $P$ is the $\SPYCT$ associated to $T\in \SPCT(\alpha)$, then $\Peak_{\uparrow}(T) = \Peak_{\leftarrow}(P)$.
\end{theorem}

\begin{lemma}\label{col12same}
Let $T\in \SPCT(\alpha)$. If the image of $T$ under reading insertion is $(P,Q)$, then the first two columns of $P$ are identical to the first two columns of $T$. 
\end{lemma}
\begin{proof}
Since entries increase along rows and up the first column of $T$, when the first entry in a row of $T$ is inserted, it is the smallest entry inserted thus far. It must therefore insert into the first column of $P$, creating a new row in $P$. Since each entry in the second column of $T$ is smaller than every entry strictly above it (by $\SPCT3$ and $\SPCT1$), when the second entry in a row of $T$ is inserted, it is the smallest entry inserted thus far except for the first entry in that row. Hence, it inserts into the new row of $P$ (in the second column) that was created by the insertion of the first entry of the corresponding row of $T$. By \cite[Lemma 3.5]{AHM}, when an increasing sequence of integers is inserted, the new box created by insertion of each integer is strictly right of the new box created by insertion of the previous integer. In particular, since entries increase along rows of $T$, when an entry of $T$ is inserted, the resulting new box of $P$ is in a column weakly right of the column of that entry in $T$. So, since entries in the first two columns of $T$ insert to the first two columns of $P$, no entry is ever bumped into the first or second column of $P$. Hence the first two columns of $T$ and $P$ are the same.
\end{proof}

\begin{proposition}\label{prop:TisSPYCT}
If $T\in \SPCT(\alpha)$, then the filling $P$ obtained by applying the reading insertion map to $T$ is an $\SPYCT$.
\end{proposition}
\begin{proof}
Since $T\in \SPCT(\alpha)$, which is a subset of $\SIT(\alpha)$, we know $P$ is an $\SYCT$ by Theorem~\ref{thm:insertionofT}. That $P$ is an $\SPYCT$ then follows immediately from Lemma \ref{col12same}, since $\SPCT3$ and $\SPYCT3$ are the same condition, and whether this condition is satisfied or not is determined by the first two columns. 
\end{proof}
Note that since every $\SPYCT$ has peak shape, Proposition \ref{prop:TisSPYCT} also implies that the $\DIRT$ obtained by inserting the reading word of $T\in \SPCT(\alpha)$ is of peak shape. Now we show the converse of Proposition~\ref{prop:TisSPYCT}. 

\begin{proposition}\label{prop:notspctnotspyct}
If $T\in \SIT(\alpha)$ but $T\notin \SPCT(\alpha)$, then the filling $P$ obtained by applying the reading insertion map to $T$ is not an $\SPYCT$. 
\end{proposition}
\begin{proof}
Note that the first column of $T$ is preserved under reading insertion, i.e., the first column of $P$ is identical to the first column of $T$. The same argument in the proof of Lemma~\ref{col12same} applies; for the first column this only depends on $T$ being an $\SIT$.

Since $T\notin \SPCT(\alpha)$, $T$ fails $\SPCT3$. 
Let $m$ be the smallest entry of $T$ such that $\SPCT3$ is violated on $\{1, \ldots , m\}$. Since entries of $T$ increase along rows, $m$ is in the first column. 
Let $k$ be the entry immediately below $m$ in the first column of $T$. If $k=m-1$, then since the first column is preserved on insertion, we have $m-1$ immediately below $m$ in the first column of $P$. Then $P$ fails $\SPYCT3$ on $\{1, \ldots , m\}$, since the entry immediately right of $m-1$ must be greater than $m$.

Therefore, suppose $k<m-1$. Let $r$ denote the row containing $k$. Since entries increase along rows of $T$, and since $T$ fails $\SPCT3$ on the entries $\{1, \ldots , m\}$, every entry in row $r$ of $T$ other than $k$ is larger than $m$. It suffices to show that in $P$, either there is no entry immediately right of $k$ or the entry immediately right of $k$ is larger than $m$, as then $P$ fails $\SPYCT3$, also on the entries $\{1, \ldots , m\}$.
Let $s$ be an entry in the second column of $T$, strictly below row $r$. We claim that when $s$ is inserted, it is the smallest entry of $T$ inserted thus far, except for the entry immediately left of $s$ in $T$. To see this, note that if there was any first column entry $j$ in $T$ strictly below $m$ and strictly above $s$ such that $j<s$, then $T$ would fail $\SPCT3$ on $\{1,\ldots, j\}$, and by $\SPCT2$, $j<m$, a contradiction. Since $s$ is smaller than every first column entry of $T$ strictly above it in $T$, by $\SPCT1$, $s$ is smaller than every entry of $T$ strictly above it. 

Insert entries from all rows of $T$ down to (and including) row $r$. Since all these entries except $k$ are greater than $m$, at this point in the process either there is no entry immediately right of $k$ in $P$, or the entry immediately right of $k$ in $P$ is greater than $m$. Now consider inserting the entries from a row below row $r$. The first entry inserts in the first column, and then the second entry must insert in the second column immediately right of the first entry since it is the smallest entry of $T$ (other than the first entry of that row) inserted thus far. The boxes of $P$ created by insertion of all later entries in this row are in the third column or later, by \cite[Lemma 3.5]{AHM}. In particular, the insertion of any entry of $T$ that is strictly below row $r$ cannot create or change the entry immediately right of $k$ in $P$. Therefore in $P$, either there is no entry immediately right of $k$, or the entry immediately right of $k$ is greater than $m$. Either way, $P$ fails $\SPYCT3$ on $\{1, \ldots , m\}$.
\end{proof}

\begin{proof}[Proof of Theorem \ref{bijection}]
By Theorem~\ref{thm:insertionofT}, reading insertion gives a bijection between $\SIT(\alpha)$ and pairs $(P,Q)$ where $P$ is an $\SYCT$ and $Q$ is a $\DIRT$ of the same shape as $P$ with row strip shape $\alpha^{\rev}$. Moreover, if reading insertion sends $T\in \SIT(\alpha)$ to $(P,Q)$, then $\Des_{\uparrow}(T) = \Des_{\leftarrow}(P)$. 
Combining this with Propositions~\ref{prop:TisSPYCT} and \ref{prop:notspctnotspyct}, we have a bijection between $\SPCT(\alpha)$ and pairs $(P,Q)$ where $P$ is an $\SPYCT$ and $Q$ is a $\DIRT$ of the same shape as $P$ with row strip shape $\alpha^{\rev}$. Since $\Des_{\uparrow}(T) = \Des_{\leftarrow}(P)$ when reading insertion sends $T\in \SIT(\alpha)$ to $(P,Q)$, we have $\Peak_{\uparrow}(T) = \Peak_{\leftarrow}(P)$.
\end{proof}

\begin{theorem}\label{thm:expansion}
Let $\alpha$ be a peak composition of $n$. The expansion of the quasisymmetric Schur $Q$-function $\tilde{Q}_\alpha$ in the peak Young quasisymmetric Schur functions is given by 
\[\tilde{Q}_\alpha = \sum c_{\alpha,\beta}\tilde{S}_\beta\] 
where the sum is over peak compositions $\beta$ of $n$, and $c_{\alpha,\beta}$ is the number of $\DIRT$s of shape $\beta$ with row strip shape $\alpha^{\rev}$. 
\end{theorem}
\begin{proof}
This is similar to the proof of \cite[Theorem 1.1]{AHM}. Let $Y(\alpha) = \{(P,Q)\}$, where $P$ is an $\SPYCT$ and $Q$ is a $\DIRT$ with the same shape as $P$ and row strip shape $\alpha^{\rev}$.  Theorem \ref{bijection} gives a bijection from $\SPCT(\alpha)$ to $Y(\alpha)$, which preserves peak sets. Therefore 
\[\sum_{T\in \SPCT(\alpha)}K_{\comp_n(\Peak_{\uparrow}(T))} = \sum_{(P,Q)\in Y(\alpha)} K_{\comp_n(\Peak_{\leftarrow}(P))}.\] 

By Theorem \ref{qsqfpeaks}, the left hand side of this equation is $\tilde{Q}_\alpha$, and by Definition~\ref{def:PYQS} the right hand side is $\sum c_{\alpha,\beta}\tilde{S}_\beta$, where $c_{\alpha,\beta}$ counts the $\DIRT$s of peak shape $\beta$ and row strip shape $\alpha^{\rev}$. 
\end{proof}

\begin{example}
Continuing with Example \ref{expansion}, since inserting the reading words generates a single $\DIRT$ of shape $(3,2)$ and a single $\DIRT$ of shape $(2,3)$, we have \[\tilde{Q}_{(3,2)} = \tilde{S}_{(3,2)}+\tilde{S}_{(2,3)}.\]
\end{example}

We have shown that the coefficients in the expansion of $\tilde{Q}_\alpha$ into peak Young quasisymmetric Schur $Q$-functions $\tilde{S}_\beta$ are determined by the $\DIRT$s of peak shape with row strip shape $\alpha^{\rev}$. Hence, if we can find all such $\DIRT$s, we obtain the expansion without needing to perform insertion. To this end we present the following procedure, which is a modified version of an analogous procedure from \cite{AHM}.

\begin{procedure}\label{proc:generateDIRTs}
Let $\alpha = (\alpha_1,\ldots ,\alpha_k)$ be a peak composition of $n$.  
Given $1\le i< k$ and a $\DIRT$ $Q$ of peak shape with row strip shape $(\alpha_{k-i+1},\ldots ,\alpha_k)^{\rev}$, we form $\DIRT$s of peak shape with row strip shape $(\alpha_{k-i},\ldots ,\alpha_k)^{\rev}$ as follows. Let $m$ be the largest element of $Q$. We assign entries from $[m+1, m+\alpha_{k-i}]$ sequentially to $Q$ in each possible way according to the rules 
\begin{enumerate}
\item The entries $m+1$ and $m+2$ become the first two entries of a new bottom row of $Q$. (Note $m+2$ exists, since $\alpha_{k-i}\ge 2$.) 
\item Each subsequent entry is placed at the end of a row strictly to the right of the last entry placed.
\item No entry can be placed at the end of a row of length $j$ if there exists a row of length $j+1$ below this row. 
\end{enumerate}
Now, starting with the single $\DIRT$ of peak shape with row strip shape $(\alpha_k)^{\rev}$, namely,
\[\begin{array}{c@{\hskip \cellsize}c@{\hskip \cellsize}c@{\hskip \cellsize}c@{\hskip \cellsize}c@{\hskip \cellsize}c}\tableau{1&2&...&\alpha_k}\end{array},\]
iterate this process on all $\DIRT$s created at each step, until the entry $n$ has been assigned to each.
\end{procedure}

\begin{proposition}
Let $\alpha$ be a peak composition of $n$.  Procedure~\ref{proc:generateDIRTs} generates all $\DIRT$s of peak shape with row strip shape $\alpha^{\rev}$.
\end{proposition}
\begin{proof}
Given a $\DIRT$ $Q$ of peak shape with row strip shape $(\alpha_{k-i+1},\ldots ,\alpha_k)^{\rev}$, any filling $Q'$ obtained from $Q$ by applying one step of Procedure~\ref{proc:generateDIRTs} will, by construction, satisfy $\DIRT1$, $\DIRT2$ and $\DIRT3$. Moreover, (3) ensures $\DIRT4$ holds, (1) ensures $Q'$ has peak shape, and $Q'$ has row-strip shape $(\alpha_{k-i},\ldots ,\alpha_k)^{\rev}$ by construction. Hence this procedure generates $\DIRT$s of peak shape with row strip shape $\alpha^{\rev}$. 

Note that (1) is necessary: since all row strips must start in the first column, if the second entry of some row strip is not placed immediately right of the first entry, then the second column of the resulting tableau will have fewer boxes strictly below row $k$ than the first column does, and thus the resulting tableau cannot have peak shape. As in \cite{AHM}, one can then use induction on $\ell(\alpha)$ to show the Procedure~\ref{proc:generateDIRTs} generates all $\DIRT$s of peak shape with row strip shape $\alpha^{\rev}$. 
\end{proof}

Note also that, by construction, the $\DIRT$s produced at each step in Procedure~\ref{proc:generateDIRTs} are distinct. 

\begin{example}
We apply Procedure~\ref{proc:generateDIRTs} to $\alpha = (3,2,3)$. Since $\alpha_3 = 3$, we start with the $\DIRT$ \[\begin{array}{c@{\hskip \cellsize}c@{\hskip \cellsize}c@{\hskip \cellsize}c@{\hskip \cellsize}c@{\hskip \cellsize}c}\tableau{1&2&3}.\end{array}\] 
At the next step, we assign entries $\{4,5\}$. By (1), there is only one way to do this, giving 
\[\begin{array}{c@{\hskip \cellsize}c@{\hskip \cellsize}c@{\hskip \cellsize}c@{\hskip \cellsize}c@{\hskip \cellsize}c}\tableau{1&2&3\\4&5&}\end{array}\]
At the third and final step, we assign entries $\{6,7,8\}$ in all possible ways. By (1), the $6$ and $7$ start a new third row, and then there are three choices for placing the $8$, giving \[\begin{array}{c@{\hskip \cellsize}c@{\hskip \cellsize}c@{\hskip \cellsize}c@{\hskip \cellsize}c@{\hskip \cellsize}c}\tableau{1&2&3\\4&5\\6&7&8}   ~~~~,~~~~ ~~~~\tableau{1&2&3\\4&5&8\\6&7}~~~~~~,~~~~~\tableau{1&2&3&8\\4&5\\6&7}\end{array}\] 
Therefore, by Theorem~\ref{thm:expansion}, we have 
\[\tilde{Q}_{(3,2,3)} = \tilde{S}_{(3,2,3)}+\tilde{S}_{(2,3,3)}+\tilde{S}_{(2,2,4)}.\]
\end{example}

As a consequence of Procedure~\ref{proc:generateDIRTs}, we obtain the following characterisation for when a quasisymmetric Schur $Q$-function is equal to a peak Young quasisymmetric Schur function.

\begin{proposition}
The quasisymmetric Schur $Q$-function indexed by a peak composition $\alpha = (\alpha_1,\ldots ,\alpha_k)$ is precisely the peak Young quasisymmetric Schur function indexed by $\alpha$ if and only if either $\alpha = (2,\ldots ,2, \alpha_k)$ or if $\alpha = (2,\ldots,2,\alpha_{k-1},1)$. 
\end{proposition}
\begin{proof} 
If $\alpha$ has one of these forms, it is immediate from (1) and (2) that the only $\DIRT$ of row strip shape $\alpha^{\rev}$ that can be obtained via Procedure~\ref{proc:generateDIRTs} is the $\DIRT$ whose row strips are rows. Since this $\DIRT$ has shape $\alpha$, we have $\tilde{Q}_\alpha = \tilde{S}_\alpha$.

If $\alpha$ does not have one of the forms described, then for some $1\le i <k$, we have $\alpha_{i+1} \ge 2$ and $\alpha_i \ge 3$. For any $\DIRT$ of row strip shape $(\alpha_{i+2}, \ldots , \alpha_k)^{\rev}$ generated by Procedure~\ref{proc:generateDIRTs}, in the next step of the procedure place all entries corresponding to $\alpha_{i+1}$ as a new single row at the bottom, then in the following step, start by placing the first two entries corresponding to $\alpha_i$ as the first two entries in a new row at the bottom. Now, to complete this step we may place remaining entries corresponding to $\alpha_i$ in either this new row or the row above, hence generating more than one $\DIRT$. So $\tilde{Q}_\alpha \neq \tilde{S}_\alpha$.
\end{proof}

\section*{Acknowledgements}
The authors acknowledge the support of the Marsden Fund, administered by the Royal Society of New Zealand Te Ap{\=a}rangi.

\bibliographystyle{abbrv} 
\bibliography{ArxivVersion}

\begin{thebibliography}{10}

\bibitem{AHM}
E.~E. Allen, J.~Hallam, and S.~K. Mason.
\newblock Dual immaculate quasisymmetric functions expand positively into
  {Y}oung quasisymmetric {S}chur functions.
\newblock {\em J. Combin. Theory Ser. A}, 157:70--108, 2018.

\bibitem{BBSSZ:dualimmaculate}
C.~Berg, N.~Bergeron, F.~Saliola, L.~Serrano, and M.~Zabrocki.
\newblock A lift of the {S}chur and {H}all-{L}ittlewood bases to
  non-commutative symmetric functions.
\newblock {\em Canad. J. Math.}, 66(3):525--565, 2014.

\bibitem{CNO}
S.-I. Choi, S.-Y. Nam, and Y.-T. Oh.
\newblock Quasisymmetric {S}chur ${Q}$-functions and peak {Y}oung
  quasisymmetric {S}chur functions.
\newblock {\em preprint}, 2024.
\newblock {\sf arXiv:2405.05867}.

\bibitem{Gessel}
I.~M. Gessel.
\newblock Multipartite {$P$}-partitions and inner products of skew {S}chur
  functions.
\newblock In {\em Combinatorics and algebra ({B}oulder, {C}olo., 1983)},
  volume~34 of {\em Contemp. Math.}, pages 289--317. Amer. Math. Soc.,
  Providence, RI, 1984.

\bibitem{HLMvW11:QS}
J.~Haglund, K.~Luoto, S.~Mason, and S.~van Willigenburg.
\newblock Quasisymmetric {S}chur functions.
\newblock {\em J. Combin. Theory Ser. A}, 118(2):463--490, 2011.

\bibitem{Jing.Li}
N.~Jing and Y.~Li.
\newblock A lift of {S}chur's ${Q}$-functions to the peak algebra.
\newblock {\em J. Combin. Theory Ser. A}, 135:268--290, 2015.

\bibitem{Oguz}
E.~Kantarc{\i}~O{\u{g}}uz.
\newblock A {N}ote on {J}ing and {L}i's type ${B}$ {Q}uasisymmetric {S}chur
  functions.
\newblock {\em Ann. Comb.}, 23(1):159--170, 2019.

\bibitem{LMvWbook}
K.~Luoto, S.~Mykytiuk, and S.~van Willigenburg.
\newblock {\em An introduction to quasisymmetric {S}chur functions: {H}opf
  algebras, quasisymmetric functions, and {Y}oung composition tableaux}.
\newblock Springer Briefs in Mathematics. Springer, New York, 2013.

\bibitem{Searles:0HC}
D.~Searles.
\newblock Diagram supermodules for $0$-{H}ecke--{C}lifford algebras.
\newblock {\em Preprint}, 2022.
\newblock arXiv: 2202.12022.

\bibitem{Stembridge:enriched}
J.~Stembridge.
\newblock Enriched ${P}$-partitions.
\newblock {\em Trans. Amer. Math. Soc.}, 349:763--788, 1997.

\end{thebibliography}

\end{document}